\documentclass[12pt, oneside]{article}

\usepackage{filecontents}

\usepackage{geometry, comment}
\usepackage[parfill]{parskip}
\usepackage{graphicx}
\usepackage[normalem]{ulem}

\usepackage{amsmath, amssymb, amsthm, mathrsfs}
\usepackage[hidelinks]{hyperref}
\usepackage{mathtools, xcolor, centernot, bm, setspace}

\usepackage[shortlabels]{enumitem}

\usepackage{titlesec}
\titleformat{\section}
{\centering\normalfont\bfseries}
{\thesection.}
{0.5em}
{\MakeUppercase}

\titleformat{\subsection}
{\normalfont\bfseries}
{\thesubsection.}
{0.5em}
{}

\usepackage{abstract}

\newcommand*\norm[1]{
	\vert\vert{#1}\vert\vert
}

\usepackage{tikzlings}

\usepackage{xparse}

\newcommand{\bqed}{\hfill\rule{1.5ex}{1.5ex}}

\newcommand{\C}{\mathbb{C}}
\newcommand{\N}{\mathbb{N}}

\newcommand{\R}{\mathbb{R}}

\newcommand{\abs}[1]{\vert #1 \vert}

\renewcommand{\epsilon}{\varepsilon}

\newtheorem{theorem}{Theorem}
\newtheorem*{theorem*}{Theorem}

\newtheorem{claim}{Claim}
\newtheorem{proposition}[theorem]{Proposition}
\theoremstyle{definition} 
\newtheorem{remark}[theorem]{Remark}
\newtheorem{example}[theorem]{Example}
\theoremstyle{definition} 

\font\tenrm = cmr17 at 12pt

\title{\bf \Large The hat polykite as an Iterated Function System}
\author{\normalsize Corey de Wit}
\date{}

\begin{document}
	\maketitle
	
	\begin{abstract}
		This paper describes the celebrated aperiodic hat tiling by Smith et al. [Comb. Theory 8 (2024), 6] as generated by an overlapping iterated function system. We briefly introduce and study infinite sequences of iterated function systems that converge uniformly in each component, and use this theory to model the hat tiling's associated imperfect substitution system.
	\end{abstract}
	
	\section{Introduction}
	
	Iterated function systems (IFSs) are a popular and well-documented model for self-similar processes \cite{fractalseverywhere}. In particular, they have been used to generate self-similar tilings \cite{bandt-tiling, tilinggifs}. A recent advance has extended IFS tiling theory to the general case where the open set condition may not be obeyed \cite{toptiling}. 
	
	In 2024, Smith et al. provided a solution to the `einstein' tiling problem (the existence of a single prototile which only tiles the plane aperiodically), one requiring reflections \cite{hattiling} and another avoiding them \cite{spectretiling}. Can IFS theory model these tilings?
	
	To address this, we introduce \emph{sequential IFSs}: sequences $(F_n = \{f^{(n)}_i : \R^q \to \R^q\}_{i=1}^M)_{n\in\N}$ of contractive IFSs where $f^{(n)}_i \to f_i$ uniformly for each $i$. Families of iterated function systems with similar properties have received little attention, though some literature has remarked on their application to time series forecasting (see Section \ref{section:remarks} for details). 
	
	In this paper, we use these sequences to construct tilings of $\R^q$. For a collection of tiles $T$, define the action of an IFS $F = \{f_i\}_{i=1}^M$ on $T$ by 
	\begin{equation} \label{eq:ifsontiling}
		F(T) := \bigcup_{i=1}^M \{f_i(t) \mid t\in T\} . 
	\end{equation}
	Then for a sequential IFS $(F_n)_{n\in\N}$ such that each $f^{(n)}_i$ has contraction ratio $\lambda$, we generate partial tilings by recursively applying the blowup $\lambda^{-1}F_n$ to some prototile set $P$, adjusting for overlaps at each stage. This approach extends IFS tiling theory to cases where elements of $P$ are not attractors, while still allowing us to describe the tiling's limiting fractal boundary (by recursively applying $F_n$ to $P$, see Proposition \ref{prop:supertile-convergence}). For example, consider a singleton prototile set containing the regular (unit) hexagon, and an IFS sequence
	\begin{equation}\label{eq:hexsifs}
		F_n = \Big\{ f_1^{(n)}(z) = \frac{1}{2}z, f_k^{(n)}(z) = \frac{1}{2}z + \frac{i\sqrt{3}}{2^n}e^{(k-2)i\pi/3}\Big\lfloor 3\cdot 2^{n-2} - \frac{1}{2} \Big\rfloor \mid k = 2,\ldots, 7 \Big\} ,
	\end{equation}
	where for $i=2,\ldots, 7$, $f^{(n)}_i$ places the $n$-th supertile radially such that its nearest tile to the origin is separated from the hexagon at the origin by $\max\{0,n-2\}$ hexagons. Overlaps are exact, so recursive applications of $2F_n$ yield well-defined partial tilings. These supertiles form a nested sequence, thus the tiling is the union limit. See Figure \ref{fig:hex} for the first 4 normalised supertiles. 
	
	\begin{figure}[h!] 
		\begin{minipage}{0.24\textwidth}
			\centering
			\includegraphics[scale = 0.33]{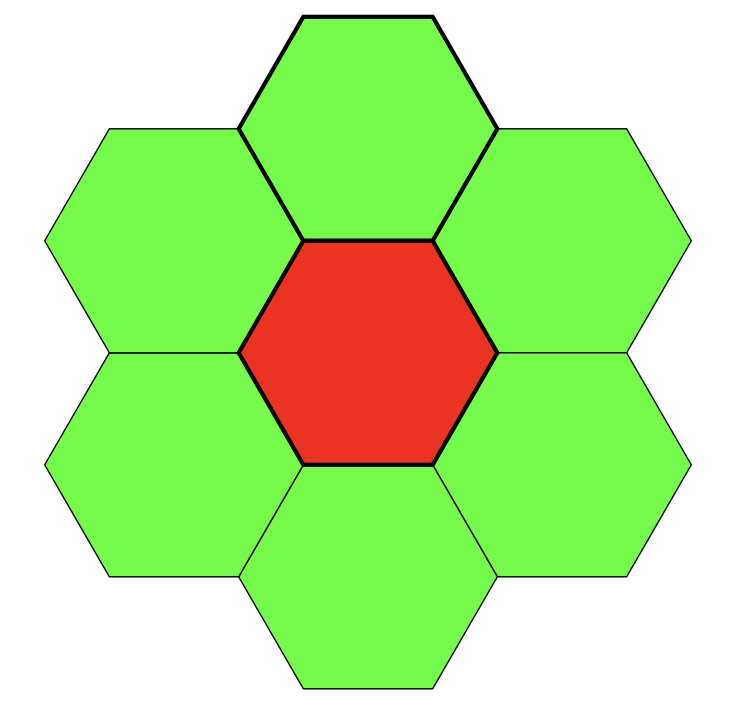}
		\end{minipage}
		\begin{minipage}{0.24\textwidth}
			\centering
			\includegraphics[scale = 0.33]{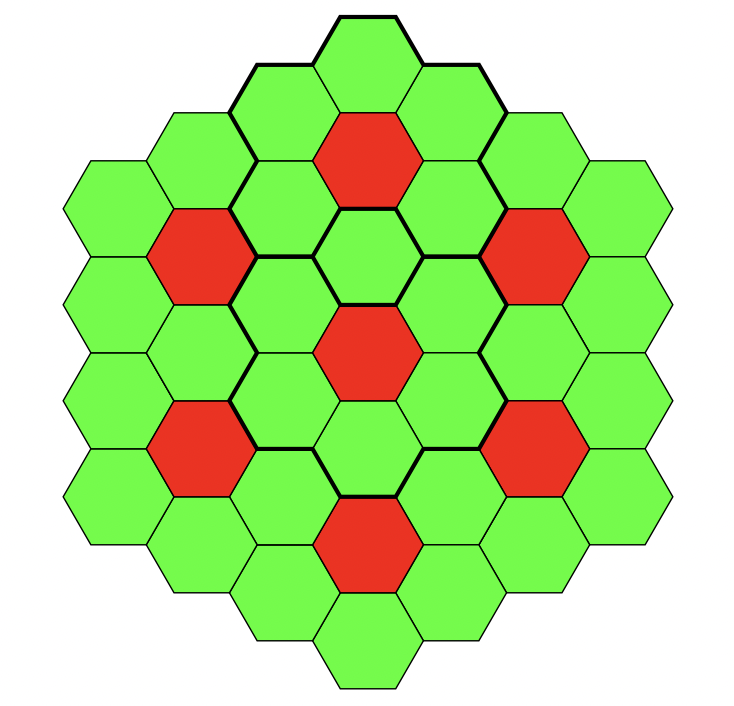}
		\end{minipage}
		\begin{minipage}{0.24\textwidth}
			\centering
			\includegraphics[scale = 0.33]{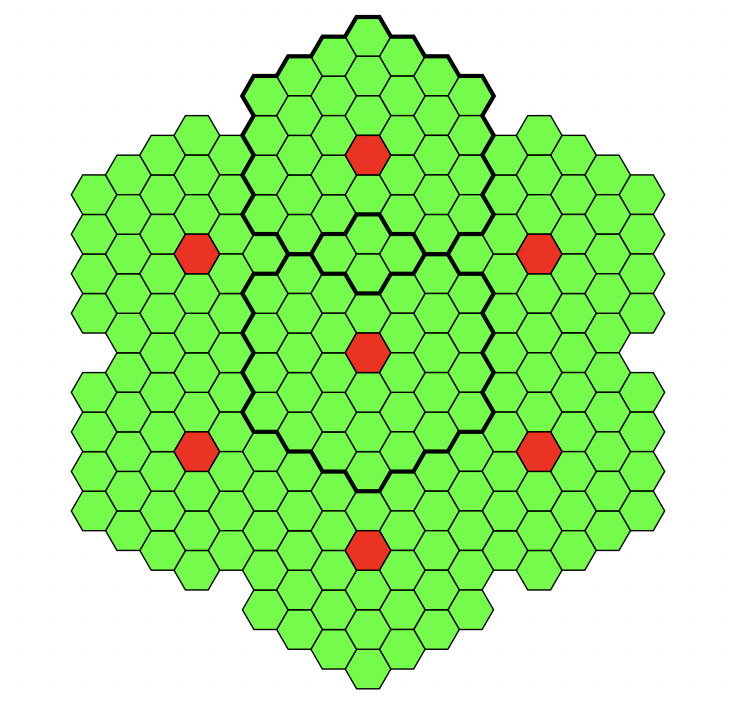}
		\end{minipage}
		\begin{minipage}{0.24\textwidth}
			\centering
			\includegraphics[scale = 0.33]{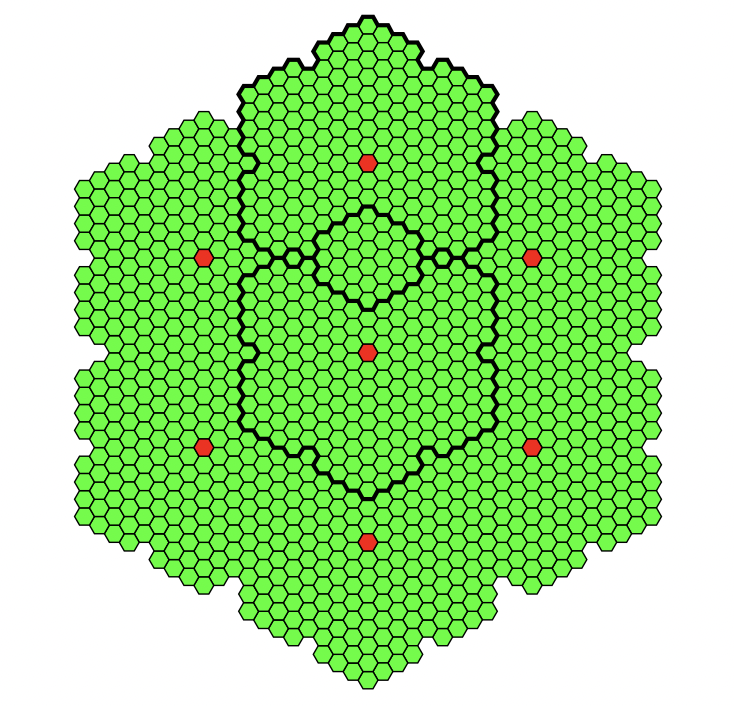}
		\end{minipage}
		\caption{First four supertiles corresponding to the SIFS \eqref{eq:hexsifs}, with the supertiles corresponding to the first and second indexed functions outlined. The limiting boundary of the tiling construction is the (hexagonal) convex hull of the supertiles above, which is also the attractor of the limit IFS. }\label{fig:hex}
	\end{figure}
	
	We will demonstrate this method's ability to model aperiodic tilings, in particular ones which are described by an imperfect substitution (where the support the substituted supertiles are not scaled copies of the original prototiles, see \cite{imperfectsub}), with the system presented in \cite[Figure 2.17]{hattiling}. 
	
	\section{Iterated Function Systems \label{section:ifs}}
	
	Let $F=\{f_i\}_{i=1}^M$ be a finite collection of contractive homeomorphisms $f_i: \R^q \to \R^q$ such that
	\[ \text{$d(f_i(x), f_i(y))\leq \lambda d(x,y)$, for some $0<\lambda<1$, for all $x,y,i$} , \]
	where $d$ is the Euclidean metric. We call $F$ an \emph{iterated function system} (IFS). We allow IFS maps to act on $\mathcal{K}$ (the set of all non-empty compact subsets of $\R^q$) by set image, and note they remain $\lambda$-Lipschitz with respect to the Hausdorff metric (which we will also denote $d$). It is well-known that $F$ possesses a unique \emph{attractor}, the only $A\in \mathcal{K}$ which obeys 
	\begin{equation} \label{eq:attractor}
		A = \bigcup_{i=1}^M f_i(A) . 
	\end{equation}
	We will need the following notions related to symbolic handling of subsets of $A$. Let $\Omega := \{1,2,\ldots, M\}^\mathbb{N}$ be the set of infinite strings of the form ${\bf j} = j_1j_2\cdots$ where each $j_i \in \{1,2,\ldots, M\}$, and $\Omega^* := \bigcup_{k=1}^\infty \{1,2,\ldots,M\}^k$ be the set of all finite strings. We define the action of ${\bf i}\in \Omega^*$ on strings by concatenation: for some ${\bf j} \in \Omega^* \cup \Omega$, define ${\bf i}{\bf j}= i_1i_2\ldots i_nj_1j_2\ldots$. The address ${\bf j}$ truncated to length $n$ is denoted by ${\bf j}\vert n = j_1j_2\ldots j_n$, and we define 
	\[ f_{{\bf j}\vert n} = f_{j_1}f_{j_2} \cdots f_{j_n} = f_{j_1}\circ f_{j_2} \circ \cdots \circ f_{j_n} . \]
	We define a metric $d'$ on $\Omega$ by $d'({\bf j},{\bf k}) = 2^{-\max \{n \vert j_m = k_m, m = 1,2,\ldots, n \}}$ for ${\bf j} \neq {\bf k}$, and note $(\Omega, d')$ is a compact metric space. Then a continuous surjection $\pi : \Omega \to A$ is defined by 
	\begin{equation} \label{eq:pi}
		\pi({\bf j}) = \lim_{m\to\infty} f_{{\bf j}\vert m}(x) ,
	\end{equation}
	which we call the \emph{coding map for $F$}. It is well-known that the limit is independent of $x$. Also, the convergence is uniform in ${\bf j}$ over $\Omega$, and uniform in $x$ over any element of $\mathcal{K}$. We say ${\bf j} \in \Omega$ is an \emph{address} of the point $\pi({\bf j}) \in A$. Since $\pi^{-1}(x)$ is closed and non-empty for all $x\in A$, a map $\tau:A\rightarrow\{  1,2,\dots,M\}  ^{\mathbb{N}}$ and a set $\Sigma$ are well-defined by
	\begin{align*}
		\tau(x)  &  :=\max\{\mathbf{k}\in\Sigma \mid \pi(\mathbf{k})=x\},\\
		\Sigma &  :=\tau(A)=\{  \tau(x):x\in A\}  ,
	\end{align*}
	where the maximum is with respect to lexicographical ordering ($i > i+1$). We call $\tau(x)$ the \textit{top address} of $x\in A$, and $\Sigma$ the \textit{top code space}. See \cite{blowup} and references therein for more details. In particular, let $\sigma : \Omega \to \Omega$ be the shift operator defined by $\sigma({\bf j}) = j_2j_3\ldots$, noting $(f_{j_1}^{-1}\circ \pi)({\bf j}) = (\pi \circ \sigma)({\bf j})$. A key property of tops code space is its shift invariance: $\sigma(\Sigma) = \Sigma$.
	
	\section{Sequential Iterated Function Systems \label{section:sifs}}
	
	Let $\mathscr{F}_M$ denote the collection of $M$-element IFSs, equipped with the metric $d_\mathscr{F}(F, G) := \max_i \norm{f_i - g_i}_\infty$. Note that $(\mathscr{F}_M, d_\mathscr{F})$ is not complete, since the (uniform) limit of a sequence of contractions may not be a contraction itself. 
	
	We call a sequence $(F_n = \{f_i^{(n)}\}_{i=1}^M)_{n\in\N} \subset \mathscr{F}_M$ a \emph{sequential IFS} (SIFS) if it converges in $(\mathscr{F}_M,d_\mathscr{F})$. Equivalently, $(F_n)_{n\in\N}$ is an SIFS when $(f_i^{(n)})_{n\in\N}$ uniformly converges to some $f_i$ for all $i$, and $F = \{f_i\}_{i=1}^M$ is itself an IFS. Let $A_n$, $A$ and $\pi_n$, $\pi$ denote the attractor and associated coding map as defined by Equations \eqref{eq:attractor} and \eqref{eq:pi} of $F_n$, $F$, respectively. We will consider SIFS's where every $F_n$ has a fixed contraction factor $\lambda$.
	
	\begin{proposition}\label{prop:attractor-convergence}
		Let $(F_n)_{n\in\N}$ be an SIFS. Then $A_n \to A$ with respect to $d$.
	\end{proposition}
	\begin{proof}
		For any $\epsilon > 0$ consider $N_i \in\N$ such that $d(f_i^{(n)}(x), f_i(x)) < \epsilon$ for all $n\geq N_i$ and $x\in\R^q$. Then by setting $N := \max\{N_i\}$, for any ${\bf j} \in \Omega$ and $k\in\N$, for all $n\geq N$
		\begin{align*}
			d\left(f^{(n)}_{{\bf j}\vert k}(x), f_{{\bf j}\vert k}(x) \right) &< d\Big(f^{(n)}_{{\bf j}\vert k}(x), \left(f_{j_1}\circ f^{(n)}_{\sigma\left({\bf j}\vert k\right)}\right)(x) \Big) + d\Big(\left(f_{j_1}\circ f^{(n)}_{\sigma\left({\bf j}\vert k\right)}\right)(x), f_{{\bf j}\vert k}(x) \Big) \\
			&< \epsilon + \lambda d\Big(f^{(n)}_{\sigma\left({\bf j}\vert k\right)}(x), f_{\sigma\left({\bf j}\vert k\right)}(x) \Big) \\
			&\;\;\vdots \\
			&<\epsilon(1 + \lambda + \cdots  + \lambda^{k-1})
		\end{align*}
		where the first inequality comes from the triangle inequality, the second inequality uses the contractivity of each $f_i$, and the final inequality is obtained by recursively applying the second inequality to its last term. Thus as $k\to\infty$, we get
		\[ d\left(\pi_n({\bf j}), \pi({\bf j})\right) < (1-\lambda)^{-1}\epsilon \]
		and the sequence of points with address ${\bf j}$ in $A_n$ converges to the point with address ${\bf j}$ (on $A$).
	\end{proof}
	
	\subsection{SIFS systems}
	
	Define a \emph{tile} as an element of $\mathcal{K}$ homeomorphic to the unit ball. For a finite set of tiles $T$, we call the triple $((F_n)_{n\in\mathbb{N}}, T, \alpha)$ an \emph{SIFS system}, where $M>\abs{T}$ and $\alpha: \{1,\ldots, M\} \to T$ is a surjection assigning each function index a `starting tile'. In particular, this system is equipped with the map $\pi_T : \Omega^* \to \mathcal{K}$ defined by
	\[ \pi_T(j_1j_2\ldots j_k) = f_{j_1}^{(k)}f_{j_2}^{(k-1)}\cdots f_{j_k}^{(1)}(\alpha(j_k)) \]
	and we say $j_1j_2\ldots j_k$ is the address of the depth-$k$ tile $\pi_T(j_1j_2\ldots j_k)$. We denote the collection of tiles at depth-$k$ as
	\begin{equation} \label{eq:supertile}
		S_k := \{ \pi_T(j_1j_2\ldots j_k) \mid j_1 j_2\ldots j_k \in\Omega^* \} ,
	\end{equation}
	which we call the \emph{$k$-th collection} (note we set $S_0 = \{T\}$), and whose support we denote $\cup S_k$. A key property of this construction is 
	\begin{equation}\label{eq:supertilerecursion}
		S_{k+1} = \bigcup_{i=1}^M f^{(k+1)}_i(S_k) .
	\end{equation}
	We note that it is not restrictive for every $F_n$ to have the same length due to the requirement for convergence. That is, if some index $i$ is only required after some $N\in\N$ we can set $f_i^{(n)}$ to the zero-map for $n<N$. Additionally if there an $N\in\N$ where some index is no longer needed, we can let $f_i^{(n)} = f_j^{(n)}$ for all $n> N$ and some non-redundant index $j$. 
	
	\begin{example}
		Any constant sequence $(F_n)_{n\in\N}$ where $F_n := F$ for some $F\in\mathscr{F}_M$ is an SIFS. Additionally, the definitions above for $((F_n)_{n\in\N}, \{A\}, \alpha)$ where $\alpha$ is the trivial map agree with the notation in Section \ref{section:ifs}. 
	\end{example}
	
	\begin{proposition} \label{prop:supertile-convergence}
		Let $((F_n)_{n\in\mathbb{N}}, T, \alpha)$ be an SIFS system. Then $\cup S_n \to A$ with respect to $d$. 
	\end{proposition}
	\begin{proof}
		Using the same $\epsilon$ and $N$ from the proof of Proposition \ref{prop:attractor-convergence}, note that
		\begin{align*}
			d\left(\pi_T({\bf j}\vert n), f_{{\bf j}\vert n}(x) \right) &= d\left(f_{j_1}^{(n)}\circ\cdots\circ f_{j_{n-N + 1}}^{(N)}(\pi_T(\sigma^{n-N+1}({\bf j})\vert (N-1))), f_{{\bf j}\vert n}(x) \right) \\
			&< d\left(f_{j_1}^{(n)}\circ\cdots\circ f_{j_{n-N + 1}}^{(N)}(\pi_T(\sigma^{n-N+1}({\bf j})\vert (N-1))) , f_{j_1}^{(n)}\circ\cdots\circ f_{j_{n-N + 1}}^{(N)}(x) \right) \\
			&\quad + d\left(f_{j_1}^{(n)}\circ\cdots\circ f_{j_{n-N + 1}}^{(N)}(x), f_{{\bf j}\vert n}(x) \right) \\
			&< \lambda^{n-N+1}\abs{A_n} + \epsilon(1+\lambda+\cdots + \lambda^{n-N}) ,
		\end{align*}
		where the first inequality comes from the triangle inequality, and the first half of the second inequality comes from the contractivity of each $f_i$. As $n\to\infty$, the first term vanishes and the second term tends to the same bound as in Proposition \ref{prop:attractor-convergence}. Thus the sequence of tiles in $S_n$ with address ${\bf j}\vert n$ tends to the point on $A$ with address ${\bf j}$. 
	\end{proof}
	
	\begin{remark}
		By the proof above, the natural extension of $\pi_T$'s domain to $\Omega^* \cup \Omega$ agrees with $\pi$, in the sense that
		\[ \pi_T({\bf j}) := \lim_{k\to\infty} \pi_T(j_1\cdots j_k) = \pi({\bf j}) . \]
	\end{remark}

	\subsection{Tiling SIFSs \label{section:tsifs}}
	
	The goal of this construction is to determine which SIFS systems produce a well-defined tiling of $\cup S_k$ for each $k$ with tiling set $\lambda^kT$. By Proposition \ref{prop:supertile-convergence}, the limiting attractor $A$ must itself be a tile in order for blowups of $k$-collections to tile $\R^q$. This implies tiles in each $k$-collection will overlap, whether only along their boundaries or something more non-trivial. We will explore a class of overlaps which allows us to extract a just-touching partial tiling from each $S_k$, and describe the blowup process to achieve a tiling with tiling set $T$. 
	
	Our process for `cutting away overlaps' to produce a just-touching tiling follows the lexicographic-based method in \cite{blowup, toptiling}. That is, we define a \emph{processed tile} to be
	\begin{equation}\label{eq:processing}
		\widetilde{\pi}_T(j_1j_2\ldots j_k) := \overline{\pi_T(j_1j_2\ldots j_k) \setminus \bigcup_{i > j} \pi_T(i_1i_2\ldots i_k)} ,
	\end{equation}
	and define the \emph{processed $k$-collection} $\widetilde{S_k}$ as the set of all tiles in $S_k$, processed (those equal to the empty set are removed). Since we wish to only tile with elements of $T$, our {\bf first condition} for SIFS systems is each processed tile must remain in $T\cup\{\emptyset\}$. 
	
	By Equation \eqref{eq:supertilerecursion} and Proposition \ref{prop:supertile-convergence}, the overlaps between the supports of shrunk $(n-1)$-collections in $S_n$ approaches the overlaps between the sets $f_i(A)$ in $A$ as $n\to\infty$ (with respect to $d$). Combining this with our first condition (which focuses on the infinitesimal scale of $S_k$), the {\bf second condition} we impose on SIFS systems is the overlap structure of each $S_k$ must match that of the attractor at depth $k$. Notice this also means the limiting IFS $F$ will be of finite type (see \cite{bandtbarnsley} and references therein for more details). We will make this more precise: let
	\[
		\Sigma_{T,k} := \{ j_1j_2\ldots j_k \in \{1,2,\ldots, M\}^k \mid \widetilde{\pi}_T(j_1j_2\ldots j_k) \in T \} 
	\]
	be the set of all tile addresses at depth-$k$ which don't vanish when processed. We require $\Sigma_{T,k} = \{{\bf j}\vert k : {\bf j} \in \Sigma\}$ for all $k$. We remark that the backward construction of this code space (i.e. copies $S_n$ for small $n$ become infinitesimally small in $S_k$ as $k\to\infty$) mirrors the generation of top addresses via the top dynamical system (see \cite{blowup}).
	
	Our process for tiling subsets of $\R^q$ with infinite volume is similar to previous IFS methods \cite{tilinggifs, toptiling}. Choose some ${\bf j}\in \Omega$. Blowups of $k$-collections about ${\bf j}$ are nested in the sense that  
	\begin{equation} \label{eq:tilinginclusion}
		\begin{aligned}
			\Big(f_{j_1}^{(1)}\Big)^{-1}\cdots \Big(f_{j_k}^{(k)}\Big)^{-1}(S_k) &=\Big(f_{j_1}^{(1)}\Big)^{-1}\cdots \Big(f_{j_{k+1}}^{(k+1)}\Big)^{-1}(f_{j_{k+1}}^{(k+1)}S_k) \\
			&\subset \Big(f_{j_1}^{(1)}\Big)^{-1}\cdots \Big(f_{j_{k+1}}^{(k+1)}\Big)^{-1}(S_{k+1}) .
		\end{aligned}
	\end{equation}
	To prove this holds upon processing, first notice the equality in Equation \eqref{eq:tilinginclusion} implies the tile with $i_1\ldots i_k$ in the blown-up $k$-collection $S_k$ is in the same position as the tile with address $j_{k+1}i_1\ldots i_k$ in the blown-up $S_{k+1}$. Additionally, when we compare their processing in Equation \eqref{eq:processing}, the condition for $s > i_1\ldots i_k$ is less restrictive than $s > j_{k+1}i_1\ldots i_k$, thus 
	\begin{equation}\label{eq:nonincseq}
		\Big(f_{j_1}^{(1)}\Big)^{-1}\cdots \Big(f_{j_k}^{(k)}\Big)^{-1} \widetilde{\pi}_T(i_1i_2\ldots i_k) \supseteq \Big(f_{j_1}^{(1)}\Big)^{-1}\cdots \Big(f_{j_k+1}^{(k+1)}\Big)^{-1} \widetilde{\pi}_T(j_{k+1}i_1i_2\ldots i_{k}) . 
	\end{equation}
	
	Finally, since $T$ is finite, the non-increasing sequence of processed tiles with address $j_m\ldots j_{k+1}i_1\ldots i_k$ in the blown-up $\widetilde{S_m}$ must stabilise. More generally, there exists a smallest $M\in\N$ such that the collection processed tiles in the blown-up $\widetilde{S_m}$ with a start of $j_m\ldots j_{k+1}$ remains the same for $m\geq M$. Let the stabilised processing of $\widetilde{S_k}$ be denoted
	\[ \widehat{S_k} := \bigg\{ \Big(f_{j_{k+1}}^{(k+1)}\Big)^{-1}\cdots \Big(f_{j_M}^{(M)}\Big)^{-1} \widetilde{\pi}_T(j_M\ldots j_{k+1}i_1i_2\ldots i_k) \mid i_1i_2\ldots i_k \in \Sigma_{T,k} \bigg\} . \]
	
	\begin{theorem}
		If an SIFS system satisfies the first and second condition outlined in this section, then 
		\[ \bigcup_{k=1}^\infty \Big(f_{j_1}^{(1)}\Big)^{-1}\cdots \Big(f_{j_k}^{(k)}\Big)^{-1}(\widehat{S_k}) \]
		is a well-defined tiling with tiling set $T$.
	\end{theorem}
	
	See \cite{reversible, toptiling} for details regarding which blowup strings ${\bf j}$ give a tiling of $\R^q$. A special case of tiling SIFSs arises when $T$ is a singleton set. In this case, the first condition forces overlaps to be exact, that is two different $k$-collections can only overlap along a common 1-skeleton. Furthermore, the second condition implies that any overlap in $S_k$ is a shrunken copy of $S_{k-1}$. Together, these properties imply the tile in $T$ is uniquely determined by $(F_n)_{n\in\N}$ and can be constructed as follows. For some $K\in\mathcal{K}$ and $F\in\mathscr{F}_M$, let
	\[ \text{Ov}_{F}(K) := \overline{\bigcup_{i,j} (f_i^{-1}f_j(K) \cap K)} \]
	be the union of all relative overlaps between sets $f_i(K)$ and $f_j(K)$ (see more on Bandt's neighbours in \cite{bandtbarnsley} and references therein). Then
	\begin{equation} \label{eq:startingtile}
		T := \lim_{n\to\infty} \text{Ov}_{F_1}\circ \cdots \circ \text{Ov}_{F_n}(A) ,
	\end{equation}
	where the limit is with respect to $d$. Following a discussion, we will give an example of such a tiling SIFS -- the class of aperiodic tiles developed by \cite{hattiling}.
	
	\subsection{General remarks} \label{section:remarks}
	
	The work in this section generalises to non-equicontractive SIFS's. This construction can be extended to cases where $T$ is countably infinite, with the domain of $\alpha$ adjusted to $\Omega^*$. In this setting, the second condition no longer implies $F$ is of finite type, and our argument for defining $\widehat{S_k}$ becomes invalid. However, \cite{toptiling} provides conditions on the blowup string which force sequences of processed tiles of the form \eqref{eq:nonincseq} to still stabilise after a finite number of steps. The situation discussed in \cite{toptiling} then corresponds to a tiling SIFS where the SIFS is constant, $T$ is the set of all possible top tiles, and $\alpha$ maps a string to the top tile associated with the corresponding $\Sigma$-cylinder set. 
	
	Additionally, Proposition \ref{prop:attractor-convergence} and \ref{prop:supertile-convergence} can be extended to self-similar measures. Denote $H_{F,p}$ and $\mu_{F,p}$ the Hutchinson operator and associated fixed point (self-similar measure), respectively, for an $F\in \mathscr{F}_M$ and probability vector $p\in\Delta^{M-1}$. Then $\mu_{F_n,p} \to \mu_{F,p}$, and letting $\mu_t = H_{F_n,p}(\mu_{t-1})$ with $\mu_0$ a probability measure with finite first moment, $\mu_t \to \mu_{F,p}$, with respect to the Monge-Kantorovich metric.
	
	We will briefly comment on other work which discusses variants of this section's introductory material. \cite{sequences1} considers SIFSs where the contraction factors of $f_i^{(n)}$ for $n\in\N$ are non-decreasing sequences for all $i = 1,2,\ldots, M$ to ensure the limiting IFS remains in $\mathscr{F}_M$. \cite{sequences2} (and references therein) analyse the sequence of self-similar measures associated with SIFSs whose IFSs are of the form $F_n := \{ \gamma_n \circ f_i \mid i\in I\}$ where $f_i$ are contractions, $(\gamma_n)_{n\in\N}$ is a convergent sequence of Lipschitz mappings, and $I$ is an uncountably infinite index set. The former queries the application of SIFSs to geographical time-dependent data, whereas the latter questions mentions its use in the theory of differential equations and time series data. 
	
	We also note that \cite{akiyama} constructs another way of tiling with overlapping imperfect substitutions, which involves assigning tiles probability weights and considering systems where the weights assigned to overlapped tiles sum to 1. This method requires both a tile set and a substitution rule, whereas the goal of this method is to determine appropriate tiling sets for a given rule (not necessarily a substitution as in Figure \ref{fig:hex}).
	
	\section{The Hat Tiling as a tiling SIFS}
	
	The hat tiling in \cite{hattiling} can be modelled by an imperfect substitution system. \emph{``At first glance, these supertiles appear to be scaled-up copies of the metatiles. If that were so, we could perhaps proceed to define a typical substitution tiling, where each scaled-up supertile is associated with a set of rigidly transformed tiles. However, none of the supertiles is similar to its corresponding metatile"}. In fact, the boundary of each supertile becomes progressively rougher. We will give a brief description of their rule and provide additional insight into their characterisation using tops. 
	
	\begin{figure}[h!] 
		\centering
		\begin{minipage}{0.4\textwidth}
			\centering
			\includegraphics[scale = 0.2]{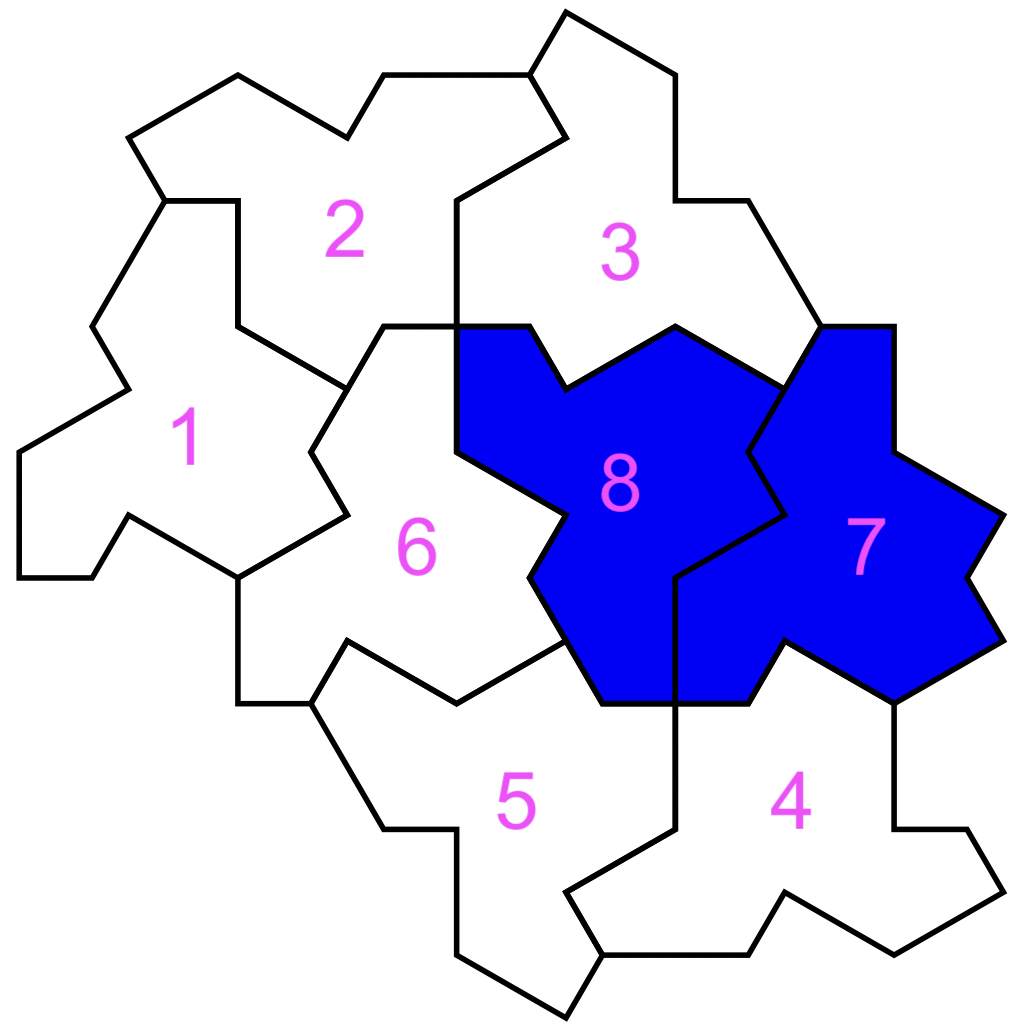}
		\end{minipage}
		\begin{minipage}{0.4\textwidth}
			\centering
			\includegraphics[scale = 0.2]{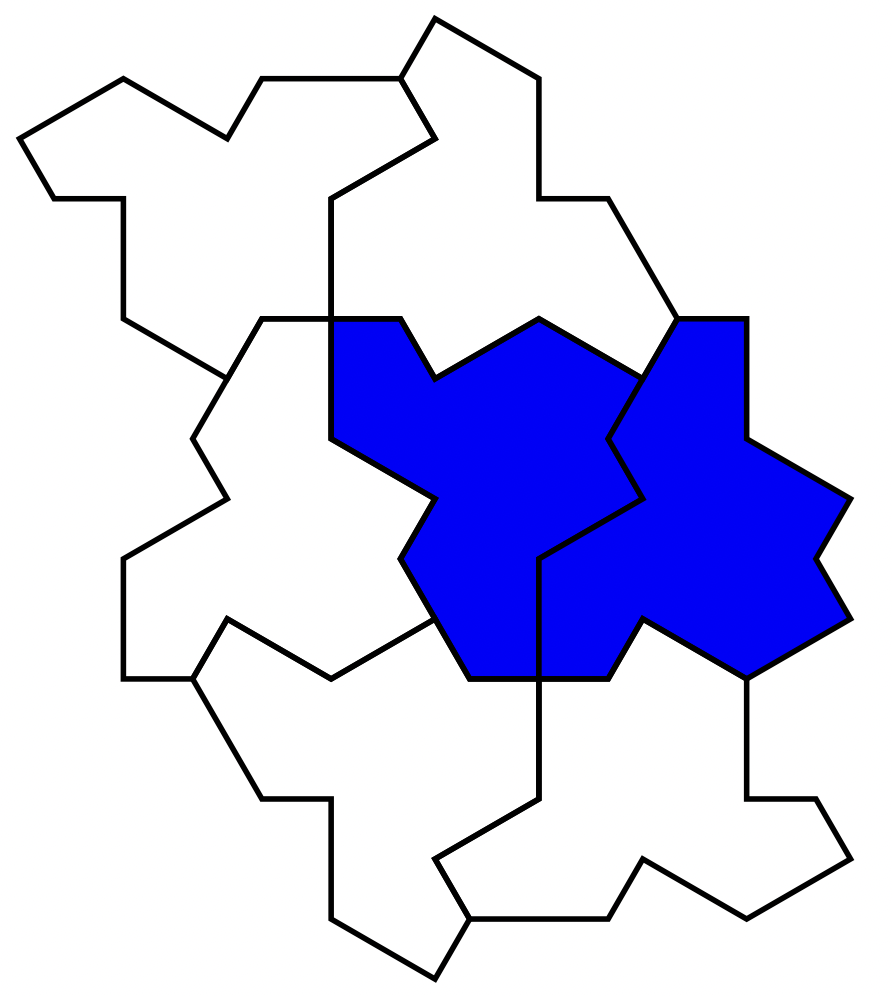}
		\end{minipage}
		\caption{$H_8$ cluster (left), $H_7$ cluster (right); $c = 1/(1+\sqrt{3})$}\label{fig:h-clusters}
	\end{figure}
	
	\begin{figure}[h!] 
		\begin{minipage}{0.24\textwidth}
			\centering
			\includegraphics[scale = 0.2]{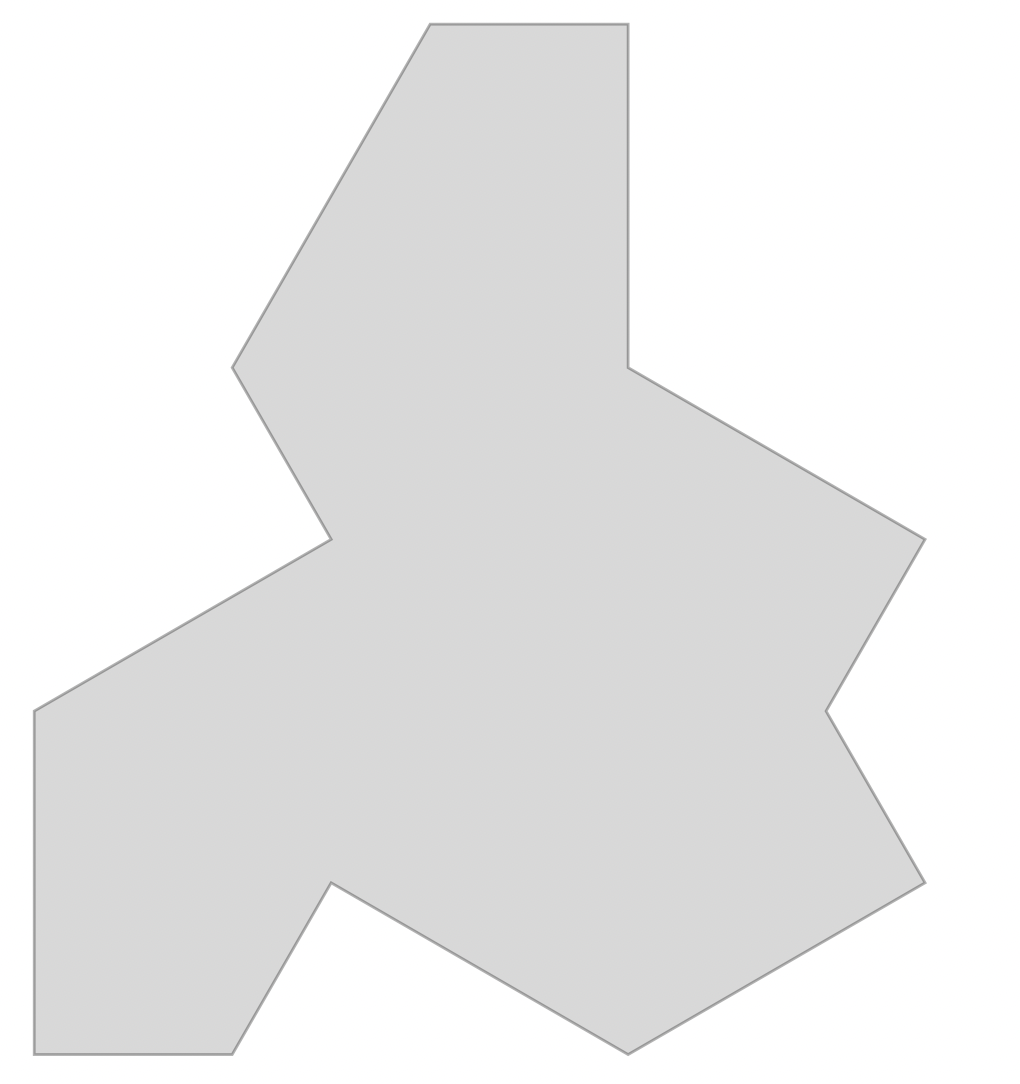}
		\end{minipage}
		\begin{minipage}{0.24\textwidth}
			\centering
			\includegraphics[scale = 0.2]{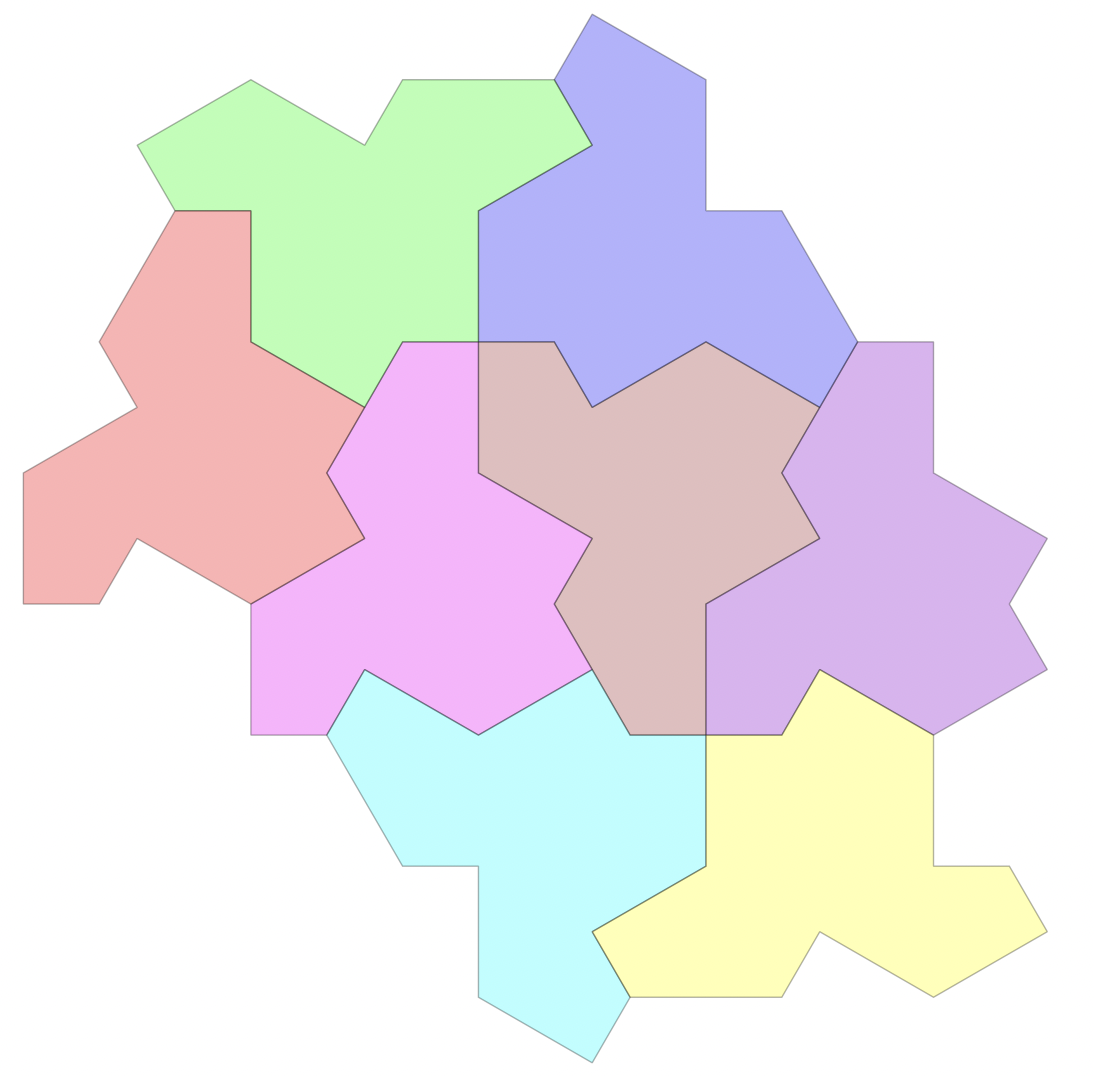}
		\end{minipage}
		\begin{minipage}{0.24\textwidth}
			\centering
			\includegraphics[scale = 0.2]{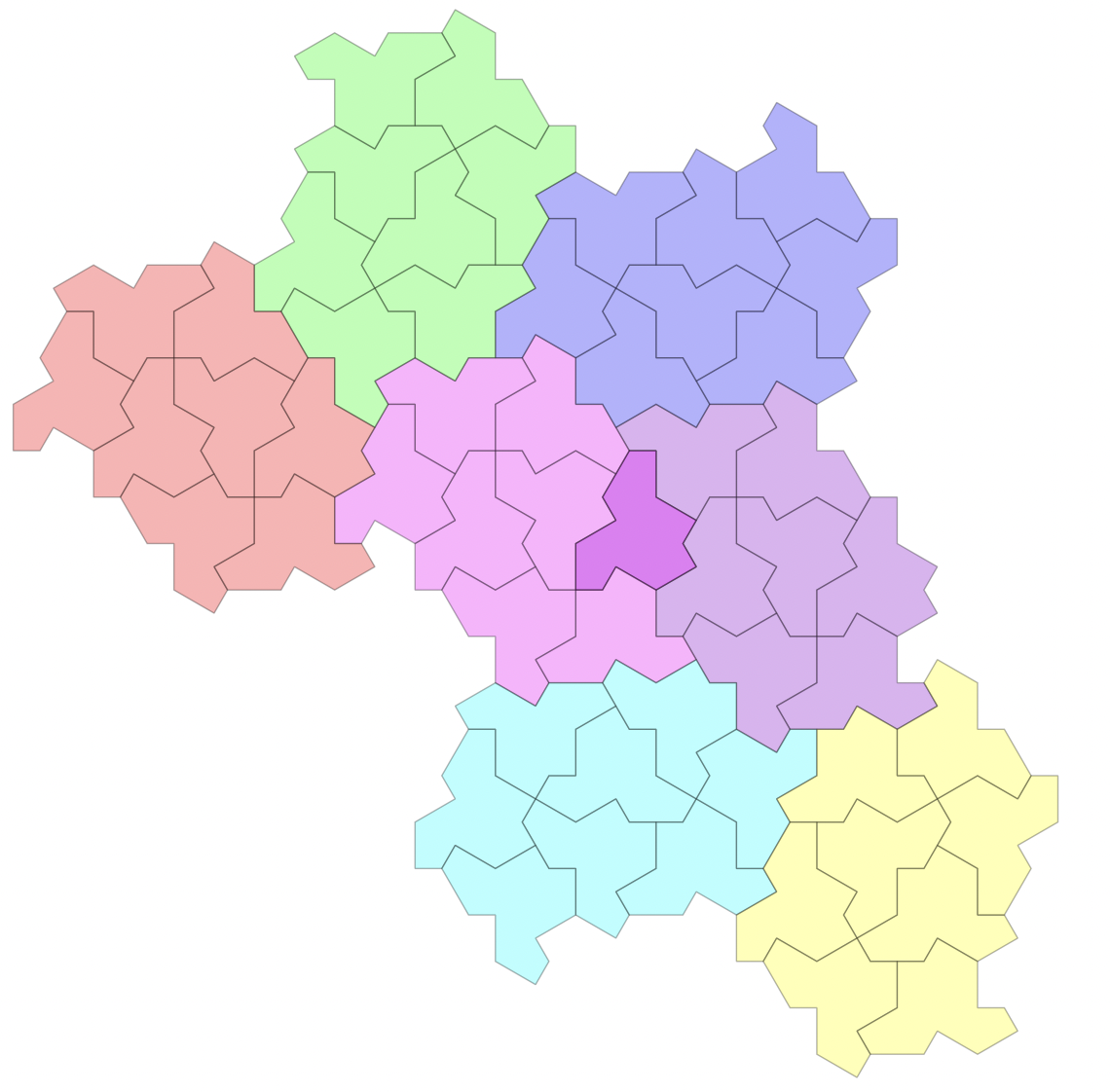}
		\end{minipage}
		\begin{minipage}{0.24\textwidth}
			\centering
			\includegraphics[scale = 0.2]{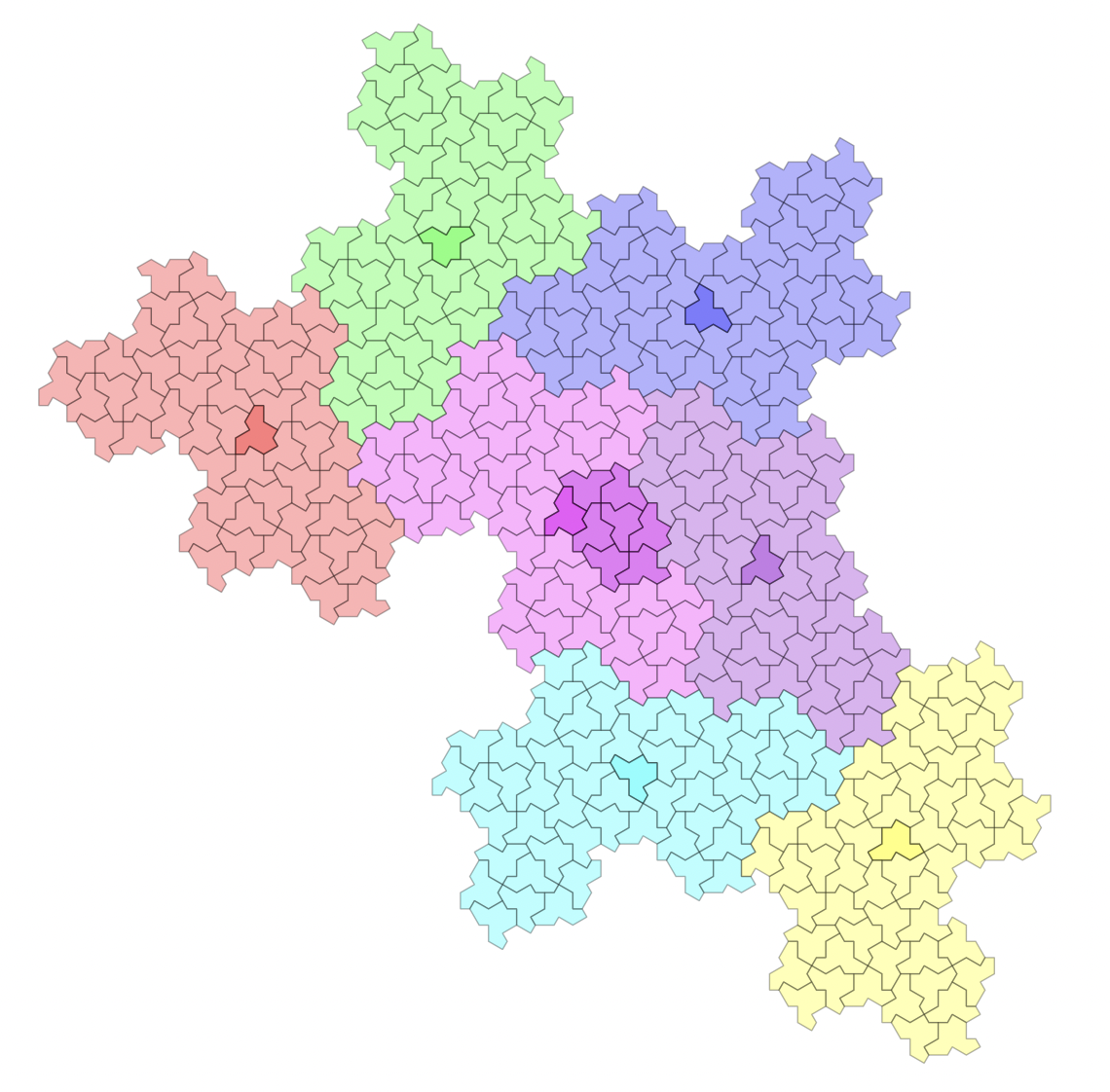}
		\end{minipage}
		\caption{$T_c$ (top left), $S_1$ (top right), $S_2$ (bottom left), $S_3$ (bottom right); $c = 1/(1+\sqrt{3})$. Each color represents a different indexed function.} \label{fig:its-hat}
	\end{figure}
	
	We will work in $\C \cong \R^2$, and use the notation in Section \ref{section:sifs}. Let $T_{c}$, for $c\in[0,1]$, denote the polygon whose vertices $\{v_i\}_{i\in\{1,2,\ldots, 13\}}$ are given by 
	\footnotesize
	\begin{align*}
		&0, 2i, (1-c)(\sqrt{3} + 3i), (1-c)(\sqrt{3} + 3i) + c(-1 + \sqrt{3}i), (1-c)(\sqrt{3} + 3i) + c(1 + 3i\sqrt{3}),\\
		&(1-c)(\sqrt{3} + 3i) + c(3 + 3i\sqrt{3}), (1-c)(\sqrt{3} + i) + c(3 + 3i\sqrt{3}), 2(1-c)\sqrt{3} + c(3 + 3i\sqrt{3}), \\
		&2(1-c)\sqrt{3} + c(2 + 2i\sqrt{3}), 2(1-c)\sqrt{3} + c(3 + \sqrt{3}i), (1-c)(\sqrt{3} - i) + c(3 + \sqrt{3}i), c(3 + \sqrt{3}i), 2c
	\end{align*}
	\normalsize
	Note this is a linear homotopy between $T_0$ (a \emph{chevron}) and $T_1$ (a \emph{comet}); Smith et al. denotes this as {\tenrm Tile$(a,b)$} where $a = 1+\sqrt{3} - b$ (so in our notation $c = a/(1+\sqrt{3})$). 
	
	Consider two clusters of $T_c$, one with 8 copies and one with 7 copies, denoted $H_8$ and $H_7$ respectively. In each cluster, two adjacent tiles are coloured, see Figure \ref{fig:h-clusters}. An imperfect substitution rule with inflation constant $\varphi^2$ (where $\varphi$ denote the golden mean) is defined as follows: each non-coloured tile is replaced with an $H_8$ cluster, and the coloured two-tile cluster is replaced with a single $H_7$ cluster (with substitutions preserving orientation of the original tile). 
	
	To convert this into an SIFS system, we will normalise each supertile and consider its creation as the single application of an IFS to the previous supertile. That is, let $\{T_c\}$ be our prototile set (thus $\alpha$ is the trivial map), and let $F_{1,c} = \{f_1^{(1)},f_2^{(1)},\ldots, f_8^{(1)}\}$ where each $f_i^{(1)}$ contracts $T_c$ by $\phi := \varphi^{-2}$ and moves it the $i$-th place in $H_8$ (with $f_1^{(1)}$ fixing the bottom left corner of $T_c$), see Figure \ref{fig:h-clusters}. Additionally, let $F_{n,c} = \{f_1^{(n)}, f_2^{(n)}, \ldots, f_7^{(n)}\}$ for $n\geq 2$ describe the mapping of contracting $S_n$ by $\phi$ and moving it to replace the $i$-th $S_{n-1}$ collection in $S_n$ (with $f_1^{(n)}$ fixing the bottom left corner of $\cup S_{n}$). 
	
	To mimic the substitution of $H_7$, we notice that $H_8 = H_7 \cup f_1^{(1)}(T_c)$. At the second level, the right most polygon in $f_6^{(2)}(\phi H_8)$ (i.e. $\pi_{\{T_c\}}(67)$) is exactly where the deleted polygon in $f_7^{(2)}(\phi H_7)$ was (i.e. $\pi_{\{T_c\}}(71)$). In general, $f_{6}^{(k)}f_7^{(k-1)}(S_{k-2})$ exactly overlaps $f_{7}^{(k)}f_1^{(k-1)}(S_{k-2})$. Thus when we process each tile, we have an equivalent characterisation of the imperfect substitution tiling. See Figure \ref{fig:its-hat} for a visualisation of the $k$-supertiles for $k=0,1,2,3$. 
	
	Finally, note this system satisfies both conditions outlined in Section \ref{section:tsifs}. In particular, $\Sigma_{T,k}$ is the set of length $k$ words which do not include $71$. Furthermore, it can be checked that the limit IFS exists (formed by taking the limit of each sequence in Theorem \ref{thm:sifs}, see Figure \ref{fig:infhat}) and the closure of its top code space is a shift of finite type (described by its one banned word $71$, see \cite{lindandmarcus} for more details). A version of this limit IFS with a change of basis can be found in \cite{toptiling}. We will now give the explicit formula for each function in the above SIFS. 
	
	\pagebreak
	
	\begin{theorem} \label{thm:sifs}
		Let $\mathsf{F}_n$ denote the $n$-th Fibonacci number (where $\mathsf{F}_1 = \mathsf{F}_2 = 1$). Then $(F_{n,c}, \{T_c\}, i \mapsto T_c)$ is a tiling SIFS, where
		\footnotesize
		\[ F_{n,c} = \left\{
		\begin{aligned}
			f_1^{(n)}(z) &= \phi z \\
			f_2^{(n)}(z) &= \phi e^{-i\pi/3}z + \phi^{n}\bigg[ \sqrt{3}(1-c)(\mathsf{F}_{2n+2}-2) + 3c(\mathsf{F}_{2n-1}-1) \bigg] \\
			&\qquad + i\phi^{n}\bigg[ 3(1-c)\mathsf{F}_{2n-1} + c\sqrt{3}(2\mathsf{F}_{2n+1}+\mathsf{F}_{2n-1}-1) \bigg] \\
			f_3^{(n)}(z) &= \phi e^{-2i\pi/3}z + \phi^{n}\bigg[ 3\sqrt{3}(1-c)(\mathsf{F}_{2n+1}-1) + 3c(2\mathsf{F}_{2n-1}+\mathsf{F}_{2n+1}-2) \bigg] \\
			&\qquad + i\phi^{n}\bigg[ -3(1-c)(\mathsf{F}_{2n-2}-1) + 3c\sqrt{3}\mathsf{F}_{2n+1} \bigg] \\
			f_4^{(n)}(z) &= \phi e^{2i\pi/3}z + \phi^{n}\bigg[ 2\sqrt{3}(1-c)\mathsf{F}_{2n+1} + 3c(2\mathsf{F}_{2n+2}-1) \bigg] \\
			&\qquad + i\phi^{n}\bigg[ -6(1-c)(\mathsf{F}_{2n+1}-1) - c\sqrt{3}(2\mathsf{F}_{2n-1}-3) \bigg] \\
			f_5^{(n)}(z) &= \phi e^{i\pi/3}z + \phi^{n}\bigg[ \sqrt{3}(1-c)(\mathsf{F}_{2n}+1) + 9c\mathsf{F}_{2n} \bigg] \\
			&\qquad + i\phi^{n}\bigg[ -3(1-c)(\mathsf{F}_{2n-1}+\mathsf{F}_{2n+1}-1) + c\sqrt{3}(-\mathsf{F}_{2n-1}-\mathsf{F}_{2n+1}+2) \bigg] \\
			f_6^{(n)}(z) &= \phi z + \phi^{n}\bigg[ \sqrt{3}(1-c)(\mathsf{F}_{2n}+\mathsf{F}_{2n-2}) + 3c(\mathsf{F}_{2n}+\mathsf{F}_{2n-2}) \bigg] \\
			&\qquad + i\phi^{n}\bigg[ -3(1-c)\mathsf{F}_{2n-1} + c\sqrt{3}\mathsf{F}_{2n-1} \bigg] \\
			f_7^{(n)}(z) &= \phi z + \phi^{n}\bigg[ \sqrt{3}(1-c)(\mathsf{F}_{2n+1}+\mathsf{F}_{2n-1}) + 3c(\mathsf{F}_{2n+1}+\mathsf{F}_{2n-1}) \bigg] \\
			&\qquad + i\phi^{n}\bigg[ -3(1-c)\mathsf{F}_{2n} + c\sqrt{3}\mathsf{F}_{2n} \bigg] \\
			f_8^{(n)}(z) &= \begin{cases} \phi\overline{z} + \phi\bigg[ \big[2\sqrt{3}(1-c) + 6c \big] + i4\sqrt{3}c\bigg] & n=1 \\ f_6^{(n)} & \text{otherwise} \end{cases}
		\end{aligned}
		\right\}
		\]
		\normalsize
	\end{theorem}

	\begin{remark}
		A similar SIFS and proof can be given for the Spectre's substitution system \cite[Figure 2.2]{spectretiling}. 
	\end{remark}
	
	\subsection{Proof of Theorem \ref{thm:sifs} \label{ap:sifsproof}}
	First note that by Banach's fixed-point theorem, our condition on $f_1$ combined with each map having a contraction factor of $\phi$ forces $f_1(z) = \phi z$. We will now make two remarks:
	
	\begin{remark} \label{rmk:intersections}
		Let the point $v_8$ relative to the tile $\pi_T(\overline{7}\vert n)$ be denoted $p_n$, and the point $v_1$ relative to the tile $\pi_T((4\overline{1})\vert n)$ be denoted $q_n$. The substitution rule forces 
		\begin{enumerate}[(1)]
			\item $f_1^{(n+1)}(\cup S_n) \cap f_2^{(n+1)}(\cup S_n) \cap f_6^{(n+1)}(\cup S_n) = \{ f_1^{(n+1)}(p_n) \} = \{ f_2^{(n+1)}(q_n) \}$
			\item $f_4^{(n+1)}(\cup S_n) \cap f_5^{(n+1)}(\cup S_n) \cap f_7^{(n+1)}(\cup S_n) = \{ f_4^{(n+1)}(p_n) \} = \{ f_5^{(n+1)}(q_n) \}$
			\item $f_2^{(n+1)}(\cup S_n) \cap f_3^{(n+1)}(\cup S_n) \cap f_6^{(n+1)}(\cup S_n) = \{ f_2^{(n+1)}(p_n) \} = \{ f_3^{(n+1)}(q_n) \}$
		\end{enumerate}
	\end{remark}
	
	To prove the theorem, we will first give an explicit formula for those points, then work backwards to calculate the translation vector for each function (i.e. where $v_1$ on $\pi_T((i\overline{1})\vert n)$, or $f_i^{(n)}(0)$, lies). By the substitution rule, each IFS in the sequence preserves angle, so this will suffice to find the explicit formula of each function. For a visualisation of the proof, see Figure \ref{fig:claim-sketch}. 
	
	\begin{figure}[h!] 
		\centering
		\includegraphics[scale = 0.8]{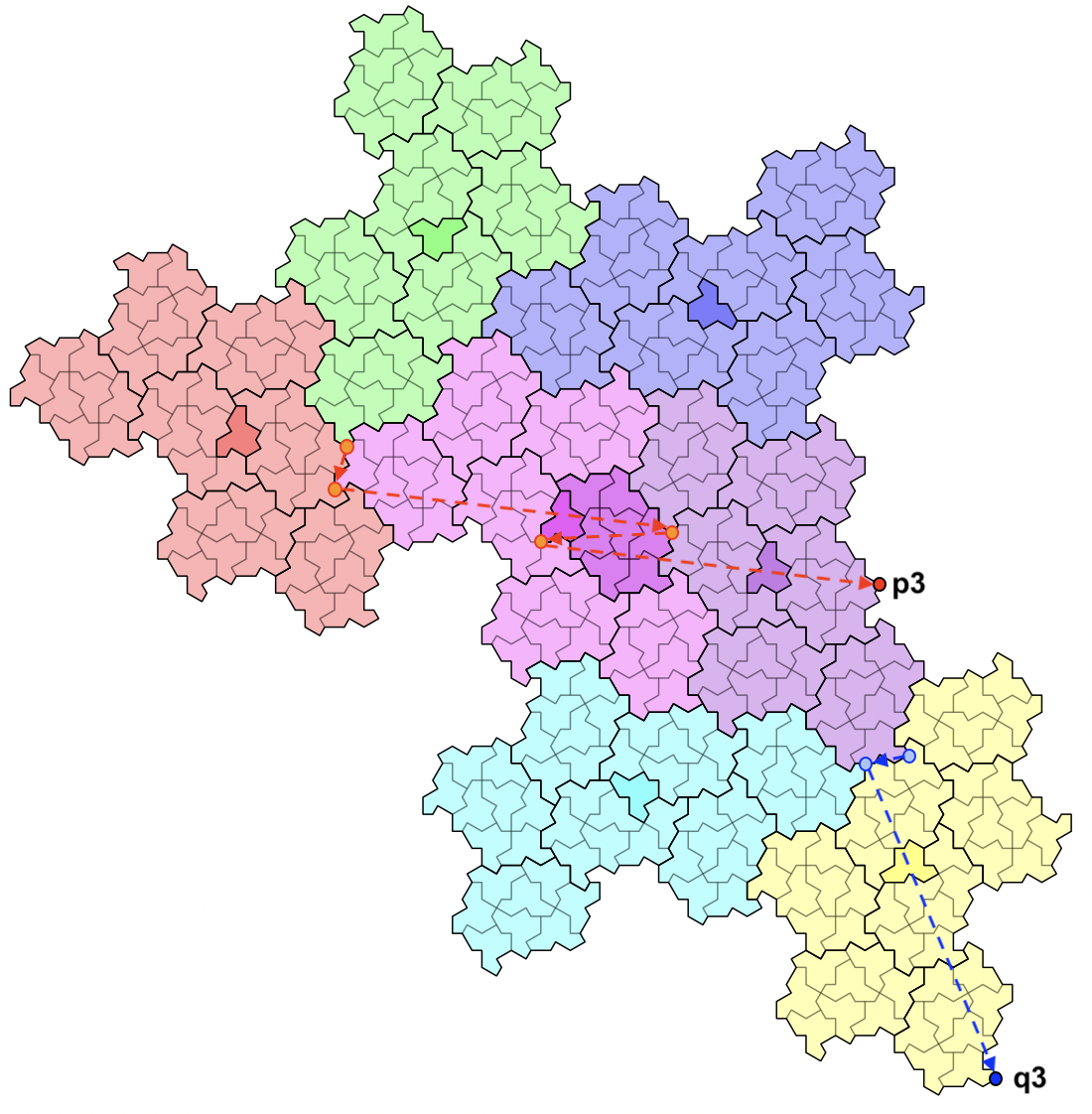}
		\caption{An example of $p_3$ and $q_3$ on $S_3$, as well as a visualisation of the proof methods for Claims \ref{claim1} and \ref{claim2}; $c = 1/(1+\sqrt{3})$}
		\label{fig:claim-sketch}
	\end{figure}
	
	\begin{claim} \label{claim1}
		\begin{align*}
			p_n = \phi^{n}\bigg[ v_8 &+ \Big[ \sqrt{3}(1-c)(\mathsf{F}_{2n+2}+\mathsf{F}_{2n} - 1) + 3c(\mathsf{F}_{2n+2}+\mathsf{F}_{2n}-1) \Big] \\
			&+ i\Big[ -3(1-c)(\mathsf{F}_{2n+1}-1) + c\sqrt{3}(\mathsf{F}_{2n+1}-1) \Big] \bigg]
		\end{align*}
	\end{claim}
	
	\emph{Proof by induction:} The base case can be shown by direct calculation. Now assume true for $1,2,\ldots,n$. By our first remark, the intersection point between $f_1^{(n+1)}(S_n)$, $f_2^{(n+1)}(S_n)$ and $f_6^{(n+1)}(S_n)$ in $S_{n+1}$ is $\phi p_n$ (first orange point). This point lies on the tile $\pi_T((6\overline{1})\vert  n+1)$, which we will denote $t$, and is the midpoint between $v_4$ and $v_5$ (relative to $t$) thus to get the expression for $v_1$ relative to $t$ (second orange point), we calculate 
	
	\footnotesize
	\begin{equation}\label{eq:claim1one}
		\begin{aligned}
			&\phi p_n - \phi^{n+1}\left( \frac{(1-c)(\sqrt{3}+3i) + c(-1+\sqrt{3}i) + (1-c)(\sqrt{3}+3i) + c(1+3i\sqrt{3})}{2}\right) \\
			&= \phi^{n+1}\Bigg( 2(1-c)\sqrt{3} + c(3+3i\sqrt{3}) + \Big[ \sqrt{3}(1-c)(\mathsf{F}_{2n+2}+\mathsf{F}_{2n} - 1) + 3c(\mathsf{F}_{2n+2}+\mathsf{F}_{2n}-1) \Big] \\
			&\qquad + i\Big[ -3(1-c)(\mathsf{F}_{2n+1}-1) + c\sqrt{3}(\mathsf{F}_{2n+1}-1) \Big] - (1-c)(\sqrt{3}+3i) - 2\sqrt{3}ic \Bigg) \\
			&= \phi^{n+1}\Bigg( \Big[ \sqrt{3}(1-c)(\mathsf{F}_{2n+2}+\mathsf{F}_{2n}) + 3c(\mathsf{F}_{2n+2}+\mathsf{F}_{2n}) \Big] + i\Big[ -3(1-c)\mathsf{F}_{2n+1} + c\sqrt{3}\mathsf{F}_{2n+1} \Big] \Bigg) ,
		\end{aligned}
	\end{equation}
	\normalsize
	
	where the first equality comes from the inductive hypothesis. We can now consider the point $p_n$ relative to $f_6^{(n+1)}(S_n)$ (the third orange point) via the inductive hypothesis by adding $\phi p_n$ to \eqref{eq:claim1one}. Now notice by the overlap rule that this is also $p_{n-1}$ relative to $f_7^{(n+1)}f_1^{(n)}(S_{n-1})$. Thus the point $v_1$ relative to $f_7^{(n+1)}(S_n)$ (the fourth orange point) is
	
	\footnotesize
	\begin{equation}\label{eq:claim1two}
		\begin{aligned}
			&\phi^{n+1}\Bigg( \Big[ \sqrt{3}(1-c)(\mathsf{F}_{2n+2}+\mathsf{F}_{2n}) + 3c(\mathsf{F}_{2n+2}+\mathsf{F}_{2n}) \Big] + i\Big[ -3(1-c)\mathsf{F}_{2n+1} + c\sqrt{3}\mathsf{F}_{2n+1} \Big] \Bigg) \\
			&\qquad + \phi p_n - \phi^2 p_{n-1} \\
			&= \phi^{n+1}\Bigg( \Big[ \sqrt{3}(1-c)(2\mathsf{F}_{2n+2}+\mathsf{F}_{2n}-\mathsf{F}_{2n-2}) + 3c(2\mathsf{F}_{2n+2}+\mathsf{F}_{2n}-\mathsf{F}_{2n-2}) \Big] \\
			&\qquad + i\Big[ -3(1-c)(2\mathsf{F}_{2n+1}-\mathsf{F}_{2n-1}) + c\sqrt{3}(2\mathsf{F}_{2n+1}-\mathsf{F}_{2n-1}) \Big] \Bigg) \\ 
			&= \phi^{n+1}\Bigg( \Big[ \sqrt{3}(1-c)(\mathsf{F}_{2n+3}+\mathsf{F}_{2n+1}) + 3c(\mathsf{F}_{2n+3}+\mathsf{F}_{2n+1}) \Big] + i\Big[ -3(1-c)\mathsf{F}_{2n+2} + c\sqrt{3}\mathsf{F}_{2n+2} \Big] \Bigg) .
		\end{aligned}
	\end{equation}
	\normalsize
	
	Finally, $p_{n+1}$ is $p_n$ relative to $f_{7}^{(n+1)}(S_n)$, thus we once again use the inductive hypothesis to add $\phi p_n$ to \eqref{eq:claim1two}:
	
	\footnotesize
	\begin{align*}
		p_{n+1} &= \phi^{n+1}\Bigg( \Big[ \sqrt{3}(1-c)(\mathsf{F}_{2n+3}+\mathsf{F}_{2n+1}) + 3c(\mathsf{F}_{2n+3}+\mathsf{F}_{2n+1}) \Big] \\
		&\qquad \qquad + i\Big[ -3(1-c)\mathsf{F}_{2n+2} + c\sqrt{3}\mathsf{F}_{2n+2} \Big] \Bigg) + \phi p_n \\
		&= \phi^{n+1}\Bigg( \Big[ \sqrt{3}(1-c)(\mathsf{F}_{2n+3}+\mathsf{F}_{2n+2} + \mathsf{F}_{2n+1} + \mathsf{F}_{2n} - 1) + 3c(\mathsf{F}_{2n+3}+\mathsf{F}_{2n+2} + \mathsf{F}_{2n+1} + \mathsf{F}_{2n} - 1) \Big] \\
		&\qquad \qquad + i\Big[ -3(1-c)(\mathsf{F}_{2n+2}+\mathsf{F}_{2n+1}-1) + c\sqrt{3}(\mathsf{F}_{2n+2}+\mathsf{F}_{2n+1} - 1) \Big] \Bigg) \\
		&= \phi^{n+1}\Bigg( \Big[ \sqrt{3}(1-c)(\mathsf{F}_{2n+4} + \mathsf{F}_{2n+2} - 1) + 3c(\mathsf{F}_{2n+4} + \mathsf{F}_{2n+2} - 1) \Big] \\
		&\qquad \qquad + i\Big[ -3(1-c)(\mathsf{F}_{2n+3}-1) + c\sqrt{3}(\mathsf{F}_{2n+3} - 1) \Big] \Bigg) .
	\end{align*}		
	\normalsize
	\bqed \\
	
	\begin{remark} 
		Now notice in Claim \ref{claim1}, we simultaneously proved (via induction) the left-most point of $f_6^{(n+1)}(S_n)$ and $f_7^{(n+1)}(S_n)$, and thus the translation vectors for $f_6^{(n)}$ and $f_7^{(n)}$ for all $n\in\N$ (namely \eqref{eq:claim1one} and \eqref{eq:claim1two} respectively). 
	\end{remark}
	
	\begin{claim} \label{claim2}
		\begin{align*}
			q_n = \phi^{n}\Bigg(\bigg[ 2\sqrt{3}(1-c)\mathsf{F}_{2n+1} + 3c(2\mathsf{F}_{2n+2}-1) \bigg] + i\bigg[ -6(1-c)(\mathsf{F}_{2n+1}-1) - c\sqrt{3}(2\mathsf{F}_{2n-1}-3) \bigg] \Bigg)
		\end{align*}
	\end{claim}
	
	\emph{Proof by induction:} The base case can be shown by direct calculation. Now assume true for $1,2,\ldots, n$. By the inductive hypothesis and our previous result, we know the expression for $q_n$ relative to $f_7^{(n+1)}(S_n)$ (the first blue point) is 
	
	\footnotesize
	\begin{align*}
			&\phi^{n+1}\Bigg(\bigg[ 2\sqrt{3}(1-c)\mathsf{F}_{2n+1} + 3c(2\mathsf{F}_{2n+2}-1) \bigg] + i\bigg[ -6(1-c)(\mathsf{F}_{2n+1}-1) - c\sqrt{3}(2\mathsf{F}_{2n-1}-3) \bigg] \Bigg) \\
			&\qquad + \phi^{n+1}\Bigg( \Big[ \sqrt{3}(1-c)(\mathsf{F}_{2n+3}+\mathsf{F}_{2n+1}) + 3c(\mathsf{F}_{2n+3}+\mathsf{F}_{2n+1}) \Big] + i\Big[ -3(1-c)\mathsf{F}_{2n+2} + c\sqrt{3}\mathsf{F}_{2n+2} \Big] \Bigg) \\
			&= \phi^{n+1}\Bigg( \Big[ \sqrt{3}(1-c)(\mathsf{F}_{2n+5}-2\mathsf{F}_{2n}) + 3c(\mathsf{F}_{2n+5}-1) \Big] \\
			&\qquad + i\Big[ -3(1-c)(\mathsf{F}_{2n+3}+\mathsf{F}_{2n+1}-2) + c\sqrt{3}(\mathsf{F}_{2n}+\mathsf{F}_{2n-2}+3) \Big] \Bigg) .
	\end{align*}
	\normalsize
	
	By our second remark, the intersection point between $f_4^{(n+1)}(S_n)$, $f_5^{(n+1)}(S_n)$ and $f_7^{(n+1)}(S_n)$ is $f_4^{(n+1)}(p_n)$ (the second blue point), which equivalently is the midpoint between $v_4$ and $v_5$ relative to the tile $\pi_T((74\overline{1})\vert n+1)$. Since the expression above represents the point $v_1$ on this tile, then the expression for $f_4^{(n+1)}(p_n)$ is 
	
	\footnotesize
	\begin{equation}\label{eq:claim2one}
		\begin{aligned}
			&\phi^{n+1}\Bigg( \Big[ \sqrt{3}(1-c)(\mathsf{F}_{2n+5}-2\mathsf{F}_{2n}) + 3c(\mathsf{F}_{2n+5}-1) \Big] + i\Big[ -3(1-c)(\mathsf{F}_{2n+3}+\mathsf{F}_{2n+1}-2) \\
			&\qquad + c\sqrt{3}(\mathsf{F}_{2n}+\mathsf{F}_{2n-2}+3) \Big] \Bigg) + e^{2i\pi/3}\bigg( (1-c)(\sqrt{3}+3i) + 2\sqrt{3}ic \bigg) \\
			&= \phi^{n+1}\Bigg( \Big[ \sqrt{3}(1-c)(\mathsf{F}_{2n+5}-2\mathsf{F}_{2n}-2) + 3c(\mathsf{F}_{2n+5}-2) \Big] + i\Big[ -3(1-c)(\mathsf{F}_{2n+3}+\mathsf{F}_{2n+1}-2) \\
			&\qquad + c\sqrt{3}(\mathsf{F}_{2n}+\mathsf{F}_{2n-2}+2) \Big] \Bigg) .
		\end{aligned}
	\end{equation}
	\normalsize
	
	Finally, since $q_{n+1} = f_4^{(n+1)}(0)$, we can find its expression via
	
	\footnotesize
	\begin{equation}\label{eq:claim2two}
		\begin{aligned}
			&\phi^{n+1}\Bigg( \Big[ \sqrt{3}(1-c)(\mathsf{F}_{2n+5}-2\mathsf{F}_{2n}-2) + 3c(\mathsf{F}_{2n+5}-2) \Big] + i\Big[ -3(1-c)(\mathsf{F}_{2n+3}+\mathsf{F}_{2n+1}-2) \\
			&\qquad + c\sqrt{3}(\mathsf{F}_{2n}+\mathsf{F}_{2n-2}+2) \Big] \Bigg) - e^{2\pi i/3} \phi p_n \\
			&= \phi^{n+1}\Bigg( \Big[ \sqrt{3}(1-c)(\mathsf{F}_{2n+5}-2\mathsf{F}_{2n}-2) + 3c(\mathsf{F}_{2n+5}-2) \Big] + i\Big[ -3(1-c)(\mathsf{F}_{2n+3}+\mathsf{F}_{2n+1}-2) \\
			&\qquad + c\sqrt{3}(\mathsf{F}_{2n}+\mathsf{F}_{2n-2}+2) \Big] \Bigg) - \phi^{n+1}\Bigg(  \Big[ \sqrt{3}(1-c)(F_{2n-1}-2) - 3c(F_{2n+2}+1) \Big] \\
			&\qquad + i\Big[ 3(1-c)F_{2n+2} + c\sqrt{3}(F_{2n+2}+2F_{2n}-1) \Big] \Bigg) \\
			&= \phi^{n+1}\Bigg( \Big[ 2\sqrt{3}(1-c)\mathsf{F}_{2n+3} + 3c(\mathsf{F}_{2n+4}-1) \Big] + i\Big[ -6(1-c)(\mathsf{F}_{2n+3}-1) + c\sqrt{3}(-2\mathsf{F}_{2n+1}+3) \Big] \Bigg) .
		\end{aligned}
	\end{equation}
	\normalsize
	\bqed\\
	
	\begin{remark}
		Now notice in Claim \ref{claim2}, we simultaneously proved (via induction) the expression for $f_4^{(n+1)}(0)$, and thus the translation vectors for $f_4^{(n)}$ for all $n\in\N$ (namely \eqref{eq:claim2two}). Furthermore, \eqref{eq:claim2one} is the point $q_n$ relative to $f_5^{(n+1)}(S_n)$, thus we can get the point $f_5^{(n+1)}(0)$, and the translation vector for $f_5^{(n)}$ for all $n\in\N$, via
		
		\footnotesize
		\begin{align*}
			&\phi^{n+1}\Bigg( \Big[ \sqrt{3}(1-c)(\mathsf{F}_{2n+5}-2\mathsf{F}_{2n}-2) + 3c(\mathsf{F}_{2n+5}-2) \Big] + i\Big[ -3(1-c)(\mathsf{F}_{2n+3}+\mathsf{F}_{2n+1}-2) \\
			&\qquad  + c\sqrt{3}(\mathsf{F}_{2n}+\mathsf{F}_{2n-2}+2) \Big] \Bigg) - e^{\pi i/3} \phi q_n \\
			&= \phi^{n+1}\Bigg( \Big[ \sqrt{3}(1-c)(\mathsf{F}_{2n+5}-2\mathsf{F}_{2n}-2) + 3c(\mathsf{F}_{2n+5}-2) \Big] + i\Big[ -3(1-c)(\mathsf{F}_{2n+3}+\mathsf{F}_{2n+1}-2) \\
			&\qquad  + c\sqrt{3}(\mathsf{F}_{2n}+\mathsf{F}_{2n-2}+2) \Big] \Bigg) - \phi^{n+1}\Bigg(\bigg[ \sqrt{3}(1-c)(4\mathsf{F}_{2n+1}-3) + 3c(\mathsf{F}_{2n+2}+\mathsf{F}_{2n-1}-2) \bigg] \\
			&\qquad + i\bigg[ 3(1-c) + c\sqrt{3}(3\mathsf{F}_{2n+2}-\mathsf{F}_{2n-1}) \bigg] \Bigg) \\
			&= \phi^{n+1}\Bigg( \Big[ \sqrt{3}(1-c)(\mathsf{F}_{2n+2}+1) + 9c\mathsf{F}_{2n+2} \Big] \\
			&\qquad + i\Big[ -3(1-c)(\mathsf{F}_{2n+3}+\mathsf{F}_{2n+1}-1) + c\sqrt{3}(-\mathsf{F}_{2n+1} -\mathsf{F}_{2n+3}+2) \Big] \Bigg) .
		\end{align*}
		\normalsize
		\bqed
	\end{remark}
	
	Finally, we are required to find $f_2^{(n)}(0)$ and $f_3^{(n)}(0)$ for all $n\in\N$. For the former, we note from Remark \ref{rmk:intersections} that $f_1^{(n+1)}(p_n) = \phi p_n = f_2^{(n+1)}(q_n)$, thus $f_2^{(n+1)}(0)$ has the expression
	
	\footnotesize
	\begin{align*}
		&\phi p_n - e^{-i\pi/3}\phi q_n \\
		= &\phi^{n+1}\bigg(\Big[ \sqrt{3}(1-c)(\mathsf{F}_{2n+2}+\mathsf{F}_{2n} + 1) + 3c(\mathsf{F}_{2n+2}+\mathsf{F}_{2n}) \Big] + i\Big[ -3(1-c)(\mathsf{F}_{2n+1}-1) \\
		&\qquad+ c\sqrt{3}(\mathsf{F}_{2n+1}+2) \Big] \bigg) - \phi^{n+1}\Bigg(\bigg[ \sqrt{3}(1-c)(-2\mathsf{F}_{2n+1}+3) + 3c(\mathsf{F}_{2n+2}-\mathsf{F}_{2n-1}+1) \bigg] \\
		&\qquad + i\bigg[ 3(1-c)(-2\mathsf{F}_{2n+1}+1) + c\sqrt{3}(-3\mathsf{F}_{2n+2}-\mathsf{F}_{2n-1}+3) \bigg] \Bigg) \\
		= &\phi^{n+1}\bigg(\Big[ \sqrt{3}(1-c)(\mathsf{F}_{2n+4} - 2) + 3c(\mathsf{F}_{2n+1} -1) \Big] + i\Big[ 3(1-c)\mathsf{F}_{2n+1} + c\sqrt{3}(2\mathsf{F}_{2n+3}+\mathsf{F}_{2n+1}-1) \Big] \bigg) .
	\end{align*}
	\normalsize
	
	Finally, from the same remark $f_2^{(n+1)}(p_n) = \phi p_n = f_3^{(n+1)}(q_n)$, thus $f_3^{(n+1)}(0)$ has the expression
	
	\footnotesize
	\begin{align*}
		&f_2^{(n+1)}(p_n) - e^{-2i\pi/3}\phi q_n \\
		= &e^{-i\pi/3}\phi^{n+1}\bigg(\Big[ \sqrt{3}(1-c)(\mathsf{F}_{2n+2}+\mathsf{F}_{2n} + 1) + 3c(\mathsf{F}_{2n+2}+\mathsf{F}_{2n}) \Big] + i\Big[ -3(1-c)(\mathsf{F}_{2n+1}-1) \\
		&\qquad+ c\sqrt{3}(\mathsf{F}_{2n+1}+2) \Big]\bigg) + \phi^{n+1}\Bigg(\bigg[ \sqrt{3}(1-c)(\mathsf{F}_{2n+4}-2) + 3c(\mathsf{F}_{2n+1}-1) \bigg] \\
		&\qquad+ i\bigg[ 3(1-c)\mathsf{F}_{2n+1} + c\sqrt{3}(2\mathsf{F}_{2n+3}+\mathsf{F}_{2n+1}-1) \bigg] \Bigg) \\
		&\qquad + \phi^{n+1}\Bigg(\bigg[ \sqrt{3}(1-c)(4\mathsf{F}_{2n+1}-3) + 3c(\mathsf{F}_{2n+2}+\mathsf{F}_{2n-1}-2) \bigg] \\
		&\qquad + i\bigg[ 3(1-c) + c\sqrt{3}(3\mathsf{F}_{2n+2}-\mathsf{F}_{2n-1}) \bigg] \Bigg) \\
		= &\phi^{n+1}\bigg(\Big[ \sqrt{3}(1-c)(-\mathsf{F}_{2n-1}+2) + 3c(\mathsf{F}_{2n+2}+1) \Big] + i\Big[ -3(1-c)\mathsf{F}_{2n+2} + c\sqrt{3}(-\mathsf{F}_{2n+2}-2\mathsf{F}_{2n}+1) \Big]\bigg) \\
		&\qquad + \phi^{n+1}\bigg(\Big[ \sqrt{3}(1-c)(\mathsf{F}_{2n+4} + 4\mathsf{F}_{2n+1} - 5) + 3c(\mathsf{F}_{2n+3} + \mathsf{F}_{2n-1} -3) \Big] \\
		&\qquad + i\Big[ 3(1-c)(\mathsf{F}_{2n+1}+1) + c\sqrt{3}(\mathsf{F}_{2n+6}-\mathsf{F}_{2n-1}-1) \Big]\bigg) \\
		= &\phi^{n+1}\bigg(\Big[ 3\sqrt{3}(1-c)(\mathsf{F}_{2n+3} - 1) + 3c(2\mathsf{F}_{2n+1} + \mathsf{F}_{2n+3} -2) \Big] + i\Big[ 3(1-c)(-\mathsf{F}_{2n} +1) + 3c\sqrt{3}\mathsf{F}_{2n+3} \Big]\bigg) .
	\end{align*}
	\normalsize
	\qed \\
	
	\begin{figure}[h!]
		\centering
		\label{fig:infhat}
		\includegraphics[scale=0.45]{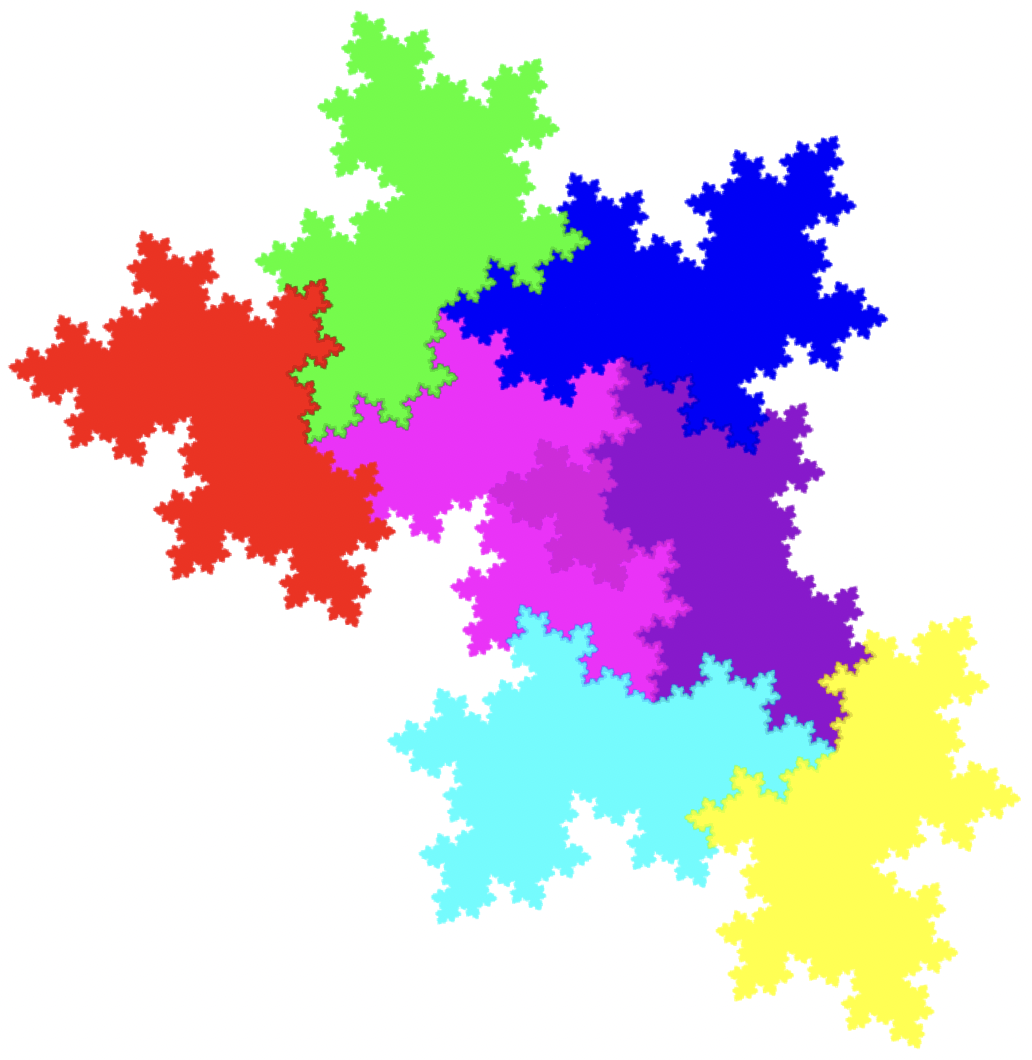}
		\caption{The attractor of the limit IFS for the SIFS given in Theorem \ref{thm:sifs}; $c = 1/(1+\sqrt{3})$. }
	\end{figure}

	\section{Acknowledgements}
	
	The author would like to thank many fruitful discussions with Michael Barnsley. In particular they credit Michael and Louisa Barnsley for the concept and calculation of the limiting hat fractal.

	\bibliography{hat-sifs}
	
\end{document}